\documentclass[12pt, twoside, a4paper]{amsart}
\usepackage{amsmath,amsthm,amsopn,amssymb,a4wide,varioref}
\usepackage{bm}
\usepackage{hyperref}
\newcommand{\comment}[1]{}

\newcommand{\bydef}{\stackrel{\rm def}{=}}

\newcommand{\clb}{\mathcal{B}}
\newcommand{\cld}{\mathcal{D}}
\newcommand{\clh}{\mathcal{H}}
\newcommand{\clk}{\mathcal{K}}

\theoremstyle{theorem}
    \newtheorem{theorem}{Theorem}
    \newtheorem{lemma}[theorem]{Lemma}
    
    \newtheorem{corollary}[theorem]{Corollary}

\theoremstyle{definition} 
    \newtheorem{definition}[theorem]{Definition}

    \newtheorem{remark}[theorem]{Remark}
    \newtheorem{example}[theorem]{Example}
    \newtheorem{exercise}[theorem]{Exercise}

\def\<{\langle}
\def\>{\rangle}

\def\bar{\overline}
\def\P{{\bf P}}

\newcommand\mnote[1]{} 
\newcommand\be{\begin{equation*}}

\newcommand\ee{\end{equation*}}

\newcommand\ben{\begin{equation}}
\newcommand\een{\end{equation}}
\newcommand\bes{\begin{eqnarray*}}
\newcommand\ees{\end{eqnarray*}}

\newcommand\bex{\begin{exercise}}
\newcommand\eex{\end{exercise}}
\newcommand\beg{\begin{example}}
\newcommand\eeg{\end{example}}
\newcommand\benu{\begin{enumerate}}
\newcommand\eenu{\end{enumerate}}
\newcommand\beit{\begin{itemize}}
\newcommand\eeit{\end{itemize}}
\newcommand\berk{\begin{remark}}
\newcommand\eerk{\end{remark}}
\newcommand\bdefn{\begin{defintion}}
\newcommand\edefn{\end{definition}}
\newcommand\bthm{\begin{theorem}}
\newcommand\ethm{\end{theorem}}
\newcommand\bprf{\begin{proof}}
\newcommand\eprf{\end{proof}}
\newcommand\blem{\begin{lemma}}
\newcommand\elem{\end{lemma}}

\newcommand{\sm}{{\raise0.3ex\hbox{$\scriptstyle \setminus$}}}

\def\CHI{\mathchoice%
{\raise2pt\hbox{$\chi$}}%
{\raise2pt\hbox{$\chi$}}%
{\raise1.3pt\hbox{$\scriptstyle\chi$}}%
{\raise0.8pt\hbox{$\scriptscriptstyle\chi$}}}
\def\smalloplus{\raise1pt\hbox{$\,\scriptstyle \oplus\;$}}

\numberwithin{equation}{section}
\newcommand{\bb}[1]{\mathbb{#1}}
\newcommand{\cl}[1]{\mathcal{#1}}

\begin{document}

\title[Admissible fundamental operators]{Admissible fundamental operators}

\author[Bhattacharyya]{Tirthankar Bhattacharyya}

\address{Department of Mathematics,
        Indian Institute of Science,
        Bangalore 560012, India}
\email{tirtha@math.iisc.ernet.in}

\author[Lata]{Sneh Lata}
\address{Department of Mathematics,
         Shiv Nadar University,
         School of Natural Sciences,
         Gautam Budh Nagar - 203207,
         Uttar Pradesh, India}
\email{sneh.lata@snu.edu.in}

\author[Sau]{Haripada Sau}
\address{Department of Mathematics,
        Indian Institute of Science,
        Bangalore 560012, India}
\email{sau10@math.iisc.ernet.in}

\thanks{This research is supported by Department of Science and
Technology, India through the project numbered SR/S4/MS:766/12 and
University Grants Commission, India via DSA-SAP.}
\date{\today}
\maketitle
\begin{abstract}
Let $F$ and $G$ be two bounded operators on two Hilbert spaces.
Let their numerical radii be no greater than one. This note
investigate when there is a $\Gamma$-contraction $(S,P)$ such that
$F$ is the fundamental operator of $(S,P)$ and $G$ is the
fundamental operator of $(S^*,P^*)$. Theorem 1 puts a necessary condition on $F$ and $G$ for them to be the fundamental operators of $(S,P)$ and  $(S^*,P^*)$ respectively. Theorem 2 shows that this necessary condition is sufficient too provided we restrict our attention to a certain special case. The general case is investigated in Theorem 3. Some of the results obtained
for $\Gamma$-contractions are then applied to tetrablock
contractions to figure out when two pairs $(F_1, F_2)$ and $(G_1,
G_2)$ acting on two Hilbert spaces can be fundamental operators of
a tetrablock contraction $(A, B, P)$ and its adjoint $(A^*, B^*,
P^*)$ respectively. This is the content of Theorem 4.
\end{abstract}
\section{Introduction}

A pair of commuting bounded operators $(S,P)$ on a Hilbert space $\clh$ having the symmetrized bidisc
$$ \Gamma = \{ (z_1 + z_2 , z_1 z_2): |z_1|, |z_2| \le 1\} = \{ \beta + \overline{\beta}p : |p| \le 1, |\beta| \le 1\}$$
as a spectral set possesses a fundamental operator $F$. Such an $(S,P)$ is called a $\Gamma$-contraction. The study of $\Gamma$-contractions was introduced and carried out very successfully over several papers by Agler and Young, see \cite{j.o.t} and the references therein.
The second component $P$ is a contraction. Let $D_P = (I - P^*P)^{1/2}$ and $\cld_P = \overline{Ran} D_P$.
The fundamental operator is the unique bounded operator on $\cld_P$ that satisfies the fundamental equation
$$ S - S^*P = D_P F D_P.$$
It has numerical radius $w(F)$ no greater than one. The fundamental operator of a $\Gamma$-contraction was introduced
in \cite{Sourav da's 1st}. The discovery of the fundamental operator of a $\Gamma$-contraction put a spurt in the activities around it. In particular, we would like to mention Sarkar's work \cite{JS} which made a significant contribution to the understanding of $\Gamma$-contractions.

The pair $(S^*, P^*)$ is also a $\Gamma$-contraction. Thus it has its own
fundamental operator $G \in \clb(\cld_{P^*})$ with $w(G) \le 1$.
 Note how both $F$ and $G$ feature in the following
explicit construction of a boundary normal dilation. The
distinguished boundary of the symmetrized bidisc is
$$ b \Gamma = \{ (z_1 + z_2 , z_1 z_2): |z_1|, |z_2| = 1\}.$$
A boundary normal dilation of a $\Gamma$-contraction $(S, P)$ is a
pair of commuting normal operators $(R, U)$ on a Hilbert space
$\clk$ containing $\clh$ such that $(R, U)$ is a $dilation$ of the
given pair $(S, P)$ and $\sigma(R, U)$, the joint spectrum is
contained in the distinguished boundary $b\Gamma$. $Dilation$
means that
$$P_{\clh} R^m U^n |_{\clh} = S^m P^n.$$
Such a pair $(R, U)$ is also called a $\Gamma$-unitary. The
following construction, done by two of the authors of the present paper in \cite{sir and me} and independently
by Pal in \cite{Sourav's single paper}, is one of the very few explicit
constructions of dilations known, the only other ones being Schaeffer's
construction of the minimal unitary dilation of a contraction in
\cite{Schaeffer} and Ando's construction of a commuting unitary
dilation of a pair of commuting bounded operators in \cite{Ando}.

\noindent \textbf{Known Theorem.} Let $(S,P)$ be
 a $\Gamma$-contraction. Let $F$ and $G$ be the fundamental
 operators of $(S,P)$ and $(S^*, P^*)$ respectively. Consider
the space  $\clk$ defined as
$$ \clk = \cdots\oplus\mathcal D_P\oplus\mathcal
D_P\oplus\mathcal D_P\oplus\mathcal H\oplus \mathcal
D_{P^*}\oplus\mathcal D_{P^*}\oplus\mathcal D_{P^*}\oplus\cdots .$$
Let $R$ and $U$ be defined on $\clk$ as follows.

\begin{equation}\label{R}
R =\left[
\begin{array}{ c c c c|c|c c c c}
\bm{\ddots}&\vdots &\vdots&\vdots   &\vdots  &\vdots& \vdots&\vdots&\vdots\\
\cdots&F&F^*&0  &0&  0&0&0&\cdots\\
\cdots&0&F&F^*  &0&  0&0&0&\cdots\\
\cdots&0&0&F  &F^*D_P&  -F^*P^*&0&0&\cdots\\ \hline

\cdots&0&0&0   &S&   D_{P^*}G&0&0&\cdots\\ \hline

\cdots&0&0&0   &0&  G^*& G&0&\cdots\\
\cdots&0&0&0   &0&  0&G^*&G&\cdots\\
\cdots&0&0&0  &0&   0& 0&G^*&\cdots\\
\vdots&\vdots&\vdots&\vdots&\vdots&\vdots&\vdots&\vdots&\bm{\ddots}\\
\end{array} \right],
\end{equation}
\begin{equation} \label{U} U = \left[
\begin{array}{ c c c c|c|c c c c}
\bm{\ddots}&\vdots &\vdots&\vdots   &\vdots  &\vdots& \vdots&\vdots&\vdots\\
\cdots&0&I&0  &0&  0&0&0&\cdots\\
\cdots&0&0&I  &0&  0&0&0&\cdots\\
\cdots&0&0&0  &D_P&  -P^*&0&0&\cdots\\ \hline

\cdots&0&0&0   &P&   D_{P^*}&0&0&\cdots\\ \hline

\cdots&0&0&0   &0&  0& I&0&\cdots\\
\cdots&0&0&0   &0&  0&0&I&\cdots\\
\cdots&0&0&0  &0&   0& 0&0&\cdots\\
\vdots&\vdots&\vdots&\vdots&\vdots&\vdots&\vdots&\vdots&\bm{\ddots}\\
\end{array} \right].
\end{equation}
 Then the pair $(R, U)$ is a $\Gamma$-unitary dilation of $(S,
 P)$.

 This shows that it is of interest to know which pair of operators $F$ and $G$, defined on different Hilbert spaces in general, satisfying $w(F) \le 1$ and $w(G) \le 1$, qualify as fundamental operators. In other words, does there always exist a $\Gamma$-contraction $(S,P)$ such that $F$ is the fundamental operator of $(S,P)$ and $G$ is the fundamental operator of $(S^*,P^*)$? In this note, our first result says that if there is such an $(S,P)$, then it forces a relation between $F$, $G$ and $P$.

For a contraction $P$ on a Hilbert space $\cl{H}$, define
$$
\Theta_P(z)=[-P+zD_{P^*}(I_\mathcal{H}-zP^*)^{-1}D_P]|_{\mathcal{D}_P} \text{ for all $z \in \mathbb{D}$}.
$$
The function $\Theta_P$ is called the {\em{characteristic function}} of the contraction $P$. By virtue of the relation $PD_P=D_{P^*}P$(see ch.1, sec.3 of \cite{Nagy-Foias}), it follows that each $\Theta_P(z)$ is an operator from $\cl{D}_P$ into $\cl{D}_{P^*}$. The characteristic function induces an operator $M_{\Theta_P}$ in $\cl{B}(H^2_{\cl{D}_P}(\mathbb{D}),H^2_{\cl{D}_{P^*}}(\mathbb{D}))$ defined by
$$
M_{\Theta_P}f(z)=\Theta_P(z)f(z) \text{ for all $z \in \mathbb{D}$}
.
$$

\begin{theorem}\label{genthm}
Let $(S,P)$ on a Hilbert space $\mathcal{H}$ be a $\Gamma$-contraction and $F,G$ be the fundamental operators of $(S,P)$ and $(S^*,P^*)$ respectively. Then
\begin{eqnarray}\label{adeq1}
{\Theta_P}(z) (F+F^*z) = (G^*+Gz) {\Theta_P}(z)
\end{eqnarray}
holds, where $\Theta_P$ is characteristic function of $P$.
\end{theorem}

Since the theorem above gives a necessary condition, it is natural to ask about sufficiency. A contraction $P$ is called $pure$ if $P^{*n}$ strongly converges to $0$ as $n$ goes to infinity. This is Arveson's terminology, see \cite{Arveson}. Sz.-Nagy and Foias called it a $C_{.0}$ contraction. The unilateral shift is a pure contraction. So are its compressions to all co-invariant subspaces.

A $\Gamma$-contraction $(S,P)$ is called pure if the contraction $P$ is pure.

\begin{theorem}\label{genthm1}
Let $P$ be a pure contraction on a Hilbert space $\mathcal{H}$. Let $F \in \mathcal{B}(\mathcal{D}_P)$ and $G \in \mathcal{B}(\mathcal{D}_{P^*})$ be two operators with numerical radius not greater than one. If (\ref{adeq1})
holds, then there exists an operator $S$ on $\mathcal{H}$ such that $(S,P)$ is a $\Gamma$-contraction and $F$,$G$ are fundamental operators of $(S,P)$ and $(S^*,P^*)$ respectively.
\end{theorem}

A contraction $P$ is called {\em{completely-non-unitary}} if it has no reducing subspaces on which its restriction is unitary.

A $\Gamma$-contraction $(S,P)$ is called completely-non-unitary if the contraction $P$ is completely-non-unitary.

Sufficiency in the situation when $P$ is not pure is more
complicated. We state it here although a couple of notations
depend on the background developed in Section 3, where the details
are given.

\begin{theorem}\label{genthm2} Let $(S,P)$ be a c.n.u. $\Gamma$-contraction on a Hilbert space $\cl H$ such that
$R=M_{e^{it}+I}$ in the representation (\ref{gammacont}) of $S$. Then
\begin{eqnarray}\label{eqn1}
\left(
\begin{array}{cc}
M_{G^*+zG} & 0\\
0 & M_{e^{it}+I}
\end{array}
\right)
\left(
\begin{array}{c}
M_{\Theta_P}\\
\Delta_P
\end{array}
\right)
&=&
\left(
\begin{array}{c}
M_{\Theta_P}\\
\Delta_P
\end{array}
\right)
M_{F+zF^*},
\end{eqnarray}
where $F\in \cl{B}(\cl D_P), \ G\in \cl{B}(\cl D_{P^*})$ are the fundamental operators
for $(S,P)$ and $(S^*,P^*)$ respectively. Moreover, if $V_1$ is as in (\ref{lastnamed}), then
\begin{eqnarray}\label{eqn2}
\left(
\begin{array}{cc}
M_{G^*+zG} & 0\\
0 & M_{e^{it}+I}
\end{array}
\right)
\left(
\begin{array}{c}
M_{V_1} M_{\Theta_P}\\
\Delta_P
\end{array}
\right)
&=&
\left(
\begin{array}{c}
M_{V_1} M_{\Theta_P}\\
\Delta_P
\end{array}
\right)
M_{Y+zY^*}\label{eqn2}\text{  holds,}
\end{eqnarray}
for some $Y\in \cl{B}(\cl D_P)$ with
$w(Y)\le 1.$

Conversely, if $P$ is a c.n.u. contraction on $\cl H$ and $F, Y\in \cl{B}(\cl D_P)$ with $w(F)\le 1, w(Y)\le 1$
and $G\in \cl{B}(\cl D_{P^*})$ with $w(G)\le 1,$ satisfy the
Equations (\ref{eqn1}) and (\ref{eqn2}), then there exists
$S\in \cl{B}(\cl H)$ so that $(S,P)$ is a c.n.u. $\Gamma$- contraction, $F$ is the fundamental operator for $(S,P)$ and $G$ is
the fundamental operator for $(S^*,P^*).$
\end{theorem}

In the last section, we discuss about when two pairs of operators can be fundamental operators of a tetrablock contraction and its adjoint. The set {\bf{tetrablock}} is defined by
$$
E = \{\underline{x}=(x_1,x_2,x_3)\in \mathbb{C}^3:
 1-x_1z-x_2w+x_3zw \neq 0 \text{ whenever }|z| < 1\text{ and }|w| < 1 \}
$$
See \cite{awy} and \cite{awy-cor} to study the geometric properties of the domain.
A commuting triple of operators $(A,B,P)$ on a Hilbert space $\mathcal{H}$ is called a tetrablock contraction if $\bar{E}$ is a spectral set. Like $\Gamma$-contractions, tetrablock contractions also possess fundamental operators and these are introduced in \cite{sir's tetrablock paper}.
Fundamental equations for a tetrablock contraction are
\begin{eqnarray}\label{Maa12}
A-B^*P=D_PF_1D_P, \text{ and }  B-A^*P=D_PF_2D_P,
\end{eqnarray}
where $D_P=(I-P^*P)^\frac{1}{2}$ is the defect operator of the contraction $P$ and $\mathcal{D}_P=\overline{Ran}D_P$ and where $F_1,F_2$ are bounded operators on $\mathcal{D}_P$.
Theorem 1.3 in \cite{sir's tetrablock paper} says that the two fundamental equations can be solved and the solutions $F_1$ and $F_2$ are unique. The unique solutions $F_1$ and $F_2$ of equations (\ref{Maa12}) are called the {\em{fundamental operators}} of the tetrablock contraction $(A,B,P)$. Moreover, $w(F_1)$ and $w(F_2)$ are not greater than $1$.

The adjoint triple $(A^*,B^*,P^*)$ is also a tetrablock contraction as can be seen from the definition. By what we stated above there are unique $G_1,G_2 \in \mathcal{B}(\mathcal{D}_{P^*})$ such that
\begin{eqnarray}\label{Maa13}
A^*-BP^*=D_{P^*}G_1D_{P^*} \text{ and } B^*-AP^*=D_{P^*}G_2D_{P^*}.
\end{eqnarray}
Moreover, $w(G_1)$ and $w(G_2)$ are not greater than $1$. A tetrablock contraction $(A,B,P)$ on a Hilbert space $\cl{H}$ is called pure tetrablock contraction, if the contraction $P$ is pure.
Along the lines of \cite{Sourav da's 2nd}, a model theory for pure tetrablock contractions was developed in \cite{sau}, using the fundamental operators. Our result for tetrablock contractions is follows.

\begin{theorem}\label{tetrathm}
Let $F_1$ and $F_2$ be fundamental operators of a tetrablock contraction $(A,B,P)$ and $G_1$ and $G_2$ be fundamental operators of the tetrablock contraction $(A^*,B^*,P^*)$. Then
\begin{eqnarray}\label{adeq3}
&&(G_1^*+G_2z)\Theta_{P}(z)=\Theta_{P}(z)(F_1+F_2^*z) \text{ and}
\\\label{help}
&&(G_2^*+G_1z)\Theta_{P}(z)=\Theta_{P}(z)(F_2+F_1^*z) \text{ holds} \text{ for all }z \in \mathbb{D}.
\end{eqnarray}
Conversely, let $P$ be a pure contraction on a Hilbert space $\cl{H}$. Let $G_1,G_2 \in \mathcal{B}(\mathcal{D}_{P^*})$ have numerical radii no greater than one and satisfy
\begin{eqnarray}\label{cond}
&&[G_1,G_2]=0 \text{ and }
[G_1,{G_1}^*]=[G_2,{G_2}^*].
\end{eqnarray}
Suppose $G_1$ and $G_2$ also satisfy Equations (\ref{adeq3}) and (\ref{help}), for some operators $F_1,F_2 \in \mathcal{B}(\mathcal{D}_{P})$ with numerical radii no greater than one. Then there exists a tetrablock contraction $(A,B,P)$ such that $F_1,F_2$ are fundamental operators of $(A,B,P)$ and $G_1,G_2$ are fundamental operators of $(A^*,B^*,P^*)$.
\end{theorem}

\section{results for pure $\Gamma$-contractions}\label{section1}

\begin{definition}
Let $\mathcal F$ and $\mathcal G$ be two Hilbert spaces. Let $F \in \mathcal B (\mathcal F)$ and $G \in \mathcal B ( \mathcal G)$. Then $(F, G)$ is called an admissible pair of operators if there is a $\Gamma$-contraction $(S, P)$ on a Hilbert space $\mathcal H$ such that $\mathcal D_P = \mathcal F$, $\mathcal D_{P^*} = \mathcal G$, $F$ is the fundamental operator of $(S, P)$ and $G$ is the fundamental operator of $(S^*, P^*)$. \end{definition}

The Hilbert spaces $H^2(\bb D)$ and $H^2(\bb T)$ are unitarily
equivalent via the map $z^n\mapsto e^{int}.$ Further, for a given
Hilbert space $\cl L, \ H^2_{\cl L}(\bb D)$ (respectively
$H^2_{\cl L} (\bb T) $) is unitarily equivalent to $\ H^2(\bb
D)\otimes \cl L(\text{respectively }\ H^2(\bb T)\otimes \cl L).$
We shall identify these unitarily equivalent spaces and use them,
without mention, interchangeably as per notational convenience.

The following useful characterization of the fundamental operator can be found in \cite{sir's tetrablock paper} (Lemma 4.1).
\begin{lemma}\label{adlem1}
Let $(S,P)$ be a $\Gamma$-contraction on a Hilbert space $\mathcal{H}$ and $F \in \mathcal{B}(\mathcal{D}_P)$ be its fundamental operator. Then $F$ is the only operator which satisfies
\begin{eqnarray}
D_PS=FD_P+F^*D_PP.
\end{eqnarray}
\end{lemma}
The next lemma gives relations between the fundamental operators of  $\Gamma$-contractions $(S,P)$ and $(S^*,P^*)$. These can be found in \cite{sir and me}(Lemma 7 and Lemma 11).
\begin{lemma}\label{adlem2}
Let $(S,P)$ be a $\Gamma$-contraction and $F$, $G$ are fundamental operators of $(S,P)$ and $(S^*,P^*)$ respectively. Then
\begin{enumerate}
\item[(a)] $PF=G^*P|_{\mathcal{D}_{P}}$ and
\\
\item[(b)] $D_{P^*}D_PF-PF^*=G^*D_{P^*}D_P-GP|_{\mathcal{D}_P}$
\end{enumerate} hold.
\end{lemma}

{\em \underline{{Proof of Theorem \ref{genthm}}.}}
For $z \in \mathbb{D}$, we have
\begin{eqnarray*}
&&\Theta_P(z)(F+F^*z)
\\
&=&[-P+\sum_{n=0}^{\infty}z^{n+1}D_{P^*}{P^*}^nD_P](F+F^*z)
\\
&=&
-PF+z(D_{P^*}D_PF-PF^*) + \sum_{n=1}^{\infty}z^{n+1}D_{P^*}{P^*}^nD_PF + \sum_{n=0}^{\infty}z^{n+2}D_{P^*}{P^*}^nD_PF^*
\\
&=&
-PF+z(D_{P^*}D_PF-PF^*)+ \sum_{n=2}^{\infty} D_{P^*}{P^*}^{n-2}(P^*D_PF+D_PF^*)
\end{eqnarray*}
\begin{eqnarray*}
&=&-PF+z(D_{P^*}D_PF-PF^*)+ \sum_{n=2}^{\infty} D_{P^*}{P^*}^{n-2}S^*D_P  \;\;\; [\text{ by Lemma \ref{adlem1}}]
\\
&=&
-PF+z(D_{P^*}D_PF-PF^*)+ \sum_{n=2}^{\infty} D_{P^*}S^*{P^*}^{n-2}D_P.
\end{eqnarray*}
And
\begin{eqnarray*}
&&(G^*+Gz)\Theta_P(z)
=
(G^*+Gz)[-P+\sum_{n=0}^{\infty}z^{n+1}D_{P^*}{P^*}^nD_P]|_{\cl{D}_P}
\\
&=&-G^*P|_{\mathcal{D}_P}+z(G^*D_{P^*}D_P-GP|_{\mathcal{D}_P}) + \sum_{n=1}^{\infty} z^{n+1}G^*D_{P^*}{P^*}^nD_P+\sum_{n=0}^{\infty}z^{n+2}GD_{P^*}{P^*}^nD_P
\end{eqnarray*}
\begin{eqnarray*}
&=&
-G^*P|_{\mathcal{D}_P}+z(G^*D_{P^*}D_P-GP|_{\mathcal{D}_P}) + \sum_{n=2}^{\infty}z^n(G^*D_{P^*}P^*+GD_{P^*}){P^*}^{n-2}D_P
\\
&=&
-G^*P|_{\mathcal{D}_P}+z(G^*D_{P^*}D_P-GP|_{\mathcal{D}_P}) + \sum_{n=2}^{\infty}z^nD_{P^*}S^*{P^*}^{n-2}D_P.
\end{eqnarray*}
Now the equality in Equation (\ref{adeq1}) follows from Lemma \ref{adlem2}.
This completes the proof. \qed

Define $W: \mathcal{H} \to H^2(\mathbb{D})\otimes \mathcal{D}_{P^*} $ by
$
W(h)=\sum_{n=0}^{\infty} z^n \otimes D_{P^*}{P^*}^nh \text{ for all $h \in \mathcal{H}$}.
$
Note that
$$
||W h||^2 = \sum_{n=0}^{\infty} ||D_{P^*}{P^*}^nh||^2 = \sum_{n=0}^{\infty} \left(||{P^*}^nh||^2 - ||{P^*}^{n+1}h||^2\right)=||h||^2 - \lim_{n \to \infty}||{P^*}^nh||^2.
$$
Therefore $W$ is an isometry in the case when $P$ is pure. It is easy to calculate that
$$
W^*(z^n \otimes \xi) = P^nD_{P^*} \xi \text{ for all $\xi \in \mathcal{D}_{P^*}$ and $n \geq 0$.}
$$
\begin{lemma}\label{L0}
For every contraction $P$, the identity
\begin{eqnarray}\label{L1}
WW^*+M_{\Theta_P}M_{\Theta_P}^*=I_{H^2(\mathbb{D})\otimes \mathcal{D}_{P^*}}
\end{eqnarray}
holds.
\end{lemma}
\begin{proof}
As observed by Arveson in the proof of Theorem 1.2 in \cite{Arveson}, the operator $W^*$ satisfies the identity
$$
W^*(k_z \otimes \xi)= (I - \bar{z}P)^{-1}D_{P^*}\xi \text{ for $z \in \mathbb{D}$ and $\xi \in \mathcal{D}_{P^*}$},
$$
where $k_z(w):=(1-\langle w,z\rangle)^{-1}$ for all $w \in \mathbb{D}$. Therefore we have
\begin{eqnarray*}
&&\langle (WW^*+M_{\Theta_P}M_{\Theta_P}^*)(k_z \otimes \xi), (k_w \otimes \eta) \rangle
\\
&=& \langle W^*(k_z \otimes \xi),W^*(k_w \otimes \eta) \rangle + \langle M_{\Theta_P}^*(k_z \otimes \xi), M_{\Theta_P}^*(k_w \otimes \eta)  \rangle
\\
&=& \langle (I - \bar{z}P)^{-1}D_{P^*}\xi, (I - \bar{w}P)^{-1}D_{P^*}\eta  \rangle + \langle k_z \otimes \Theta_P(z)^*\xi, k_w \otimes \Theta_P(w)^*\eta  \rangle
\\
&=& \langle D_{P^*}(I - wP^*)^{-1}(I - \bar{z}P)^{-1}D_{P^*}\xi, \eta \rangle + \langle k_z, k_w \rangle \langle \Theta_P(w)\Theta_P(z)^*\xi, \eta  \rangle
\\
&=& \langle k_z \otimes \xi, k_w \otimes \eta \rangle \text{ for all $z ,w \in \mathbb{D}$ and $\xi, \eta \in \mathcal{D}_{P^*}$}.
\end{eqnarray*}
Where the last equality follows from the following well-known identity
$$
I - \Theta_P(w)\Theta_P(z)^* = (1 - w\bar{z})D_{P^*}(I - wP^*)^{-1}(I - \bar{z}P)^{-1}D_{P^*}.
$$
Now using the fact that $\{k_z: z \in \mathbb{D}\}$ forms a total set of $H^2(\mathbb{D})$, the assertion follows.
\end{proof}

\underline{\textit{Proof of Theorem \ref{genthm1}}.} Since $P$ is
pure, $W$ is an isometry. We first find a relation between $P$,
$W$ and $M_z$, multiplication by the variable $z$ on
$H^2(\mathbb{D})\otimes \mathcal{D}_{P^*}$.
\begin{eqnarray}\label{contraction}
M_z^*Wh=M_z^*\left(\sum_{n=0}^{\infty}z^nD_{P^*}{P^*}^nh\right)=\sum_{n=0}^{\infty}z^nD_{P^*}{P^*}^{n+1}h=WP^*h.
\end{eqnarray}
Therefore $M_z^*W=WP^*$. Define $S$ on $\mathcal{H}$ by
$S=W^*M_{G^*+Gz}W$. Since $P$ is pure, from Lemma \ref{L0}, we
have $(RanW)^\perp=RanM_{\Theta_P}$. The equation
$M_{\Theta_P}M_{F+F^*z}=M_{G^*+Gz}M_{\Theta_P}$ implies that
$RanM_{\Theta_P}$ is invariant under $M_{G^*+Gz}$, in other words
$RanW$ is co-invariant under $M_{G^*+Gz}$.
\begin{eqnarray*}
P^*S^*&=&W^*M_z^*WW^*M_{G^*+Gz}^*W
\\
&=&
W^*M_z^*M_{G^*+Gz}^*W \;\;\;[\text{ since $WW^*$ is a projection onto $RanW$.}]
\\
&=&
W^*M_{G^*+Gz}^*M_z^*W\;\;\;[\text{ since $M_z$ and $M_{G^*+Gz}$ commute.}]
\\
&=&
W^*M_{G^*+Gz}^*WW^*M_z^*W=S^*P^*.
\end{eqnarray*}
Now
\begin{eqnarray*}
S^*-SP^*
&=&
W^*M_{G^*+Gz}^*W-W^*M_{G^*+Gz}WW^*M_z^*W
\\
&=&
W^*(I \otimes G + M_z^* \otimes G^*)W-W^*(I \otimes G^* + M_z \otimes G)(M_z^* \otimes I)W
\\
&=&
W^*(I \otimes G + M_z^* \otimes G^*)W-W^*(M_z^* \otimes G^* + M_zM_z^* \otimes G)W
\\
&=&
W^*(P_{\mathbb{C}} \otimes G)W\;\;\;[\mbox{$P_{\mathbb{C}}$ is the projection of $H^2(\mathbb{D})$ onto constants.}]
\\
&=&
D_{P^*}GD_{P^*}.
\end{eqnarray*}
For all $\theta \in (0,2\pi]$, we have $G^*+e^{i\theta}G=e^{i\frac{\theta}{2}}(e^{-i\frac{\theta}{2}}G^*+e^{i\frac{\theta}{2}}G)$. Hence $\lVert G^*+e^{i\theta}G \rVert = \lVert (e^{-i\frac{\theta}{2}}G^*+e^{i\frac{\theta}{2}}G) \rVert.$
Note that for all $\theta \in (0,2\pi]$ and $\xi \in \mathcal{D}_{P^*}$ we have
\begin{eqnarray*}
\lvert\langle (e^{-i\frac{\theta}{2}}G^*+e^{i\frac{\theta}{2}}G)\xi, \xi\ \rangle \rvert &=& \lvert  e^{-i\frac{\theta}{2}}\langle G^*\xi,\xi\rangle + e^{i\frac{\theta}{2}}\langle G\xi,\xi \rangle \rvert
\\
&\leq&
\lvert \langle G^*\xi,\xi \rangle \rvert + \lvert \langle G\xi,\xi \rangle \rvert \leq 2. [\text{ since $w(G) \leq 1$}]
\end{eqnarray*}
Since $(e^{-i\frac{\theta}{2}}G^*+e^{i\frac{\theta}{2}}G)$ is a self adjoint operator, we have $\lVert (e^{-i\frac{\theta}{2}}G^*+e^{i\frac{\theta}{2}}G) \rVert \leq 2$. Therefore $\lVert (G^*+Gz) \rVert \leq 2$ for all $z \in \mathbb{D}$, which implies that $\lVert M_{G^*+Gz} \rVert \leq 2$. Hence $\lVert S \rVert \leq 2$.
\\
Hence $(S^*,P^*)$ is a commuting pair of operators on $\mathcal{H}$ such that the spectral radius of $S$ is not greater than two and the operator equation
$S^*-SP^*=D_{P^*}XD_{P^*}$ has a solution for X (namely $G$) with numerical radius of $X$ not greater than one. So $(S^*,P^*)$ is a $\Gamma$-contraction and hence so is $(S,P)$.
\\
Now we will show that $F$ is the fundamental operator of $(S,P)$.
Note that if $X$ is the fundamental operator of $(S,P)$, then by Theorem \ref{genthm} we have
$M_{\Theta_P}M_{X+X^*z}=M_{G^*+Gz}M_{\Theta_P}.$ Also by hypothesis we have $M_{\Theta_P}M_{F+F^*z}=M_{G^*+Gz}M_{\Theta_P}$. Since $P$ is pure contraction, $M_{\Theta_P}$ is an isometry and hence we have $M_{X+X^*z}=M_{F+F^*z}$ on $H^2_{\cl D_P}(\bb D)$. Which implies $X=F$. Therefore $F$ is the fundamental operator of $(S,P)$.
This completes the proof of the theorem. \qed

\begin{corollary}
Let $P$ be a pure contraction on a Hilbert space $\mathcal{H}$. Let $F \in \mathcal{B}(\mathcal{D}_P)$ and $G \in \mathcal{B}(\mathcal{D}_{P^*})$ be two operators with numerical radius not greater than one. If (\ref{adeq1}) holds, then the pair $(R, U)$ as defined in (\ref{R}) and (\ref{U}) is a $\Gamma$-unitary. \end{corollary}

\begin{proof} Theorem 2 says that under these assumptions, there is an $S$ on $\mathcal H$ such that $(S, P)$ is a $\Gamma$-contraction, $F$ is the fundamental operator of $(S, P)$ and $G$ is the fundamental operator of $(S^*, P^*)$. Now, the Known Theorem of the Introduction section says that $(R, U)$ is the $\Gamma$-unitary dilation of $(S, P)$. \end{proof}

\section{The general case}
In this section we shall prove Theorem  \ref{genthm2} which is a
version of Theorem \ref{genthm1} that holds for the c.n.u. case. As
we noted when Theorem \ref{genthm2} was stated, certain background
concepts need to be developed. We first recall two minimal
isometric dilations of a c.n.u. contraction. Let $P\in \cl{B}(\cl
H)$ be a c.n.u. contraction.
\begin{enumerate}
\item[(i)] Note that
$$I \geq PP^* \geq P^2{P^*}^2 \geq \cdots \geq P^n{P^*}^n \geq \cdots \geq 0.$$
Therefore there exists a positive bounded operator, say
$P_{\infty}^2$, such that $P_{\infty}^2 h= \lim_{n\to
\infty}P^nP^{*n}h$ for all $h\in \cl H.$ Then
$PP_{\infty}^2P^* = P_{\infty}^2,$ which implies that
$||P_{\infty}h|| = ||P_{\infty}P^*h||$ for all $h.$ This
defines an isometry $T\in \cl{B}(\overline{Ran(P_{\infty}}))$
such that $TP_{\infty}=P_{\infty}P^*.$ Let $U\in \cl{B}(\cl
K)$ be the minimal unitary extension of $T.$ Then $\Pi_0 : \cl
H \to H^2_{\cl D_{P^*}}(\bb D) \oplus \cl K,$ defined as
$$
\Pi_0(h)
=
\left(
\begin{array}{c}
Wh\\
P_{\infty}
\end{array}
\right),
$$
is an isometry, where $W:\cl H \to H^2_{\cl D_{P^*}}(\bb D), \ W(h)=\sum_{n=0}^\infty z^nD_{P^*}P^{*n}h.$ We can check that
$\left(
\begin{array}{cc}
M_z\otimes I & 0\\
0 & U^*
\end{array}
\right)$ is a minimal isometric dilation of $\Pi_0 P\Pi_0^*$ and
$$\Pi_0P^*
=
\left(
\begin{array}{cc}
M_z\otimes I & 0\\
0 & U^*
\end{array}
\right)^* \Pi_0.
$$
\\
\item[(ii)]
Let
$$\Theta _P(z) =
[-P + \sum_{n=0}^{\infty} z^{n+1}D_{P^*}P^{*n}D_P]|_{\cl D_P} \text{ for all }z \in \mathbb{D}$$ be the characteristic function of $P$. For all $t \in [0,2\pi)$ define the operator
$$\Delta_P(t)=[I-{\Theta_P(e^{it})}^*\Theta_P(e^{it})]^\frac{1}{2}$$
and the subspace
$$\cl S_P = \{M_{\Theta_P}f\oplus \Delta_Pf:
f\in H^2_{\cl D_P}(\bb D)\}.$$ Then $\cl S_P$ is a closed
subspace of $H^2_{\cl D_{P^*}}(\bb D)\oplus
\overline{\Delta_PL^2_{\cl D_P}(\bb T)}$.
Let $\cl Q_P$ be the orthogonal complement of $\cl S_P$ in $H^2_{\cl D_{P^*}(\bb D)}\oplus \overline{\Delta_PL^2_{\cl D_P}(\bb T)}.$

There exists an isometry $\Pi : \cl H \to H^2_{\cl D_{P^*}}(\bb D)\oplus \overline{\Delta_PL^2_{\cl D_P}(\bb T)}$ with
$\Pi(\cl H)=\cl Q_P$ such that
$\left(
\begin{array}{cc}
M_z & 0\\
0 & M_{e^{it}}
\end{array}
\right)$ is a minimal isometric dilation of $\Pi P\Pi^*$
and
\begin{eqnarray}\label{cont}
\Pi P^*
=
\left(
\begin{array}{cc}
M_z & 0\\
0 & M_{e^{it}}
\end{array}
\right)^* \Pi.
\end{eqnarray}
\end{enumerate}

Thus $\Pi$ and $\Pi_0$ give two minimal isometric dilations of
$P$. But the minimal dilation is unique up to unitary equivalence.
Thus we get a unitary $\Phi:H^2_{\cl D_{P^*}}(\bb D)\oplus
\overline{\Delta_PL^2_{\cl D_P}(\bb T)} \longrightarrow H^2_{\cl
D_{P^*}}(\bb D)\oplus \cl K,$ such that $\Phi\Pi = \Pi_0$ and
\begin{eqnarray}\label{*1}
\Phi
\left(
\begin{array}{cc}
M_z & 0\\
0 & M_{e^{it}}
\end{array}
\right)^*
=
\left(
\begin{array}{cc}
M_z\otimes I & 0\\
0 & U^*
\end{array}
\right)^*\Phi.
\end{eqnarray}
Since $\Phi$ is unitary and satisfies (\ref{*1}), by an easy
matrix calculation and the fact that any operator intertwining a
pure isometry and a unitary is zero(Lemma 2.5 in \cite{j.o.t}), we
get $\Phi$ to be of the form
\begin{eqnarray}\label{lastnamed}
\Phi
=
\left(
\begin{array}{cc}
I \otimes V_1 & 0\\
0 & V_2
\end{array}
\right)
\end{eqnarray}
 where $V_1\in \cl{B}(\cl D_{P^*})$ and $V_2\in \cl{B}(\overline{\Delta_PL^2_{\cl D_P}(\bb T)}, \cl K)$ are
 unitary operators.

\begin{lemma}\label{Bresult} Let $P$ be a c.n.u. $\Gamma$-contraction on $\cl H.$
Let $X\in \cl{B}(\cl D_{P^*}), \ w(X)\le 1$ and $R\in \cl{B}(\overline{\Delta_PL^2_{\cl D_P}(\bb T)})$ such that
$(R,M_{e^{it}})$ is a $\Gamma$-unitary on $\overline{\Delta_PL^2_{\cl D_P}(\bb T)}$.
If
\begin{equation}\label{cnu6}
\left(
\begin{array}{cc}
M_{X^*+z X} & 0\\
0 & R
\end{array}
\right)\cl S_P \subseteq \cl S_P,
\end{equation}
then there exists $Y\in \cl{B}(\cl D_P)$ with $w(Y)\le 1$ such that
\begin{equation*}
\left(
\begin{array}{cc}
M_{X^*+zX} & 0\\
0 & R
\end{array}
\right)
\left(
\begin{array}{c}
M_{\Theta_P}\\
\Delta_P
\end{array}
\right)
=
\left(
\begin{array}{c}
M_{\Theta_P}\\
\Delta_P
\end{array}
\right)
M_{Y+zY^*}.
\end{equation*}
\end{lemma}

\begin{proof} Equation (\ref{cnu6}) allows us to define an operator
$T\in \cl{B}(H_{\cl D_P}^2(\bb D))$ so that
\begin{equation}\label{cnu7}
\left(
\begin{array}{cc}
M_{X^*+zX} & 0\\
0 & R
\end{array}
\right)
\left(
\begin{array}{c}
M_{\Theta_P}\\
\Delta_P
\end{array}
\right)
=
\left(
\begin{array}{c}
M_{\Theta_P}\\
\Delta_P
\end{array}
\right)
T.
\end{equation}
In other words,
\begin{equation}
T =
\left(
\begin{array}{c}
M_{\Theta_P}\\
\Delta_P
\end{array}
\right)^*
\left(
\begin{array}{cc}
M_{X^*+zX} & 0\\
0 & R
\end{array}
\right)
\left(
\begin{array}{c}
M_{\Theta_P}\\
\Delta_P
\end{array}
\right)
\end{equation}

To prove the result, it is enough to show that $(T,M_z)$ is a $\Gamma$- isometry. Since $w(X)\le 1$, as shown in the previous section, we have $ ||M_{X^*+zX}||\le 2.$ Also, $(R,M_{e^{it}})$ is
a $\Gamma$-unitary, therefore $||R||\le 2.$ Thus, from Equation (\ref{cnu7}), we can easily deduce that $||T||\le 2,$ since the operator
$\left(
\begin{array}{c}
M_{\Theta_P}\\
\Delta_P
\end{array}
\right)$
is an isometry. We shall now show that $T$ commutes with $M_z.$

From equation (\ref{cnu7}) we have
\begin{eqnarray}
M_{X^*+zX}M_{\Theta_P} &=& M_{\Theta_P}T\label{cnu8}\\
R\Delta_P &=& \Delta_PT.\label{cnu9}
\end{eqnarray}

Note that $M_z$ commute with $M_{X^*+zX}$ and $M_{\theta_P}.$ Therefore
applying $M_z$ on both sides of Equation (\ref{cnu8}) we get

\begin{equation}\label{cnu10}
M_{\Theta_P}TM_z = M_{\Theta_P}M_zT.
\end{equation}

Also, $M_{e^{it}}|_{\overline{\Delta_PL^2_{\cl D_P}(\bb T)}}$ commutes with $R$ and $\Delta_P,$ therefore
applying $M_{e^{it}}$ on both sides of Equation (\ref{cnu9}) we get

\begin{equation}\label{cnu11}
\Delta_PTM_z = \Delta_PM_zT.
\end{equation}

Equations (\ref{cnu10}) and (\ref{cnu11}) together with the fact that
$
\left(
\begin{array}{c}
M_{\Theta_P}\\
\Delta_P
\end{array}
\right)
$
is an isometry yield $TM_z = M_zT.$

Lastly, we shall show that
$T = T^*M_z.$
To accomplish this, consider
\begin{eqnarray*}
M_z^* T
&=&
M_z^*
\left(
\begin{array}{c}
M_{\Theta_P}\\
\Delta_P
\end{array}
\right)^*
\left(
\begin{array}{cc}
M_{X^*+zX} & 0\\
0 & R
\end{array}
\right)
\left(
\begin{array}{c}
M_{\Theta_P}\\
\Delta_P
\end{array}
\right)\\
&=&
\left(
\begin{array}{c}
M_{\Theta_P}\\
\Delta_P
\end{array}
\right)^*
\left(
\begin{array}{cc}
M_z^* & 0\\
0 & M_{e^{it}}^*
\end{array}
\right)
\left(
\begin{array}{cc}
M_{X^*+zX} & 0\\
0 & R
\end{array}
\right)
\left(
\begin{array}{c}
M_{\Theta_P}\\
\Delta_P
\end{array}
\right)\\
&=& T^*.
\end{eqnarray*}
Consequently, $M_z^*T=T^*,$ that is, $T=T^*M_z.$ Therefore we can
conclude that $(T,M_z)$ is a $\Gamma$-isometry. Agler and Young
showed in \cite{j.o.t} that the only way this can happen is that
$T$ is of the form $M_{Y+zY^*}$ for some $Y\in \cl{B}(\cl D_P), \
w(Y)\le 1$. This completes the proof.
\end{proof}

The next result, apart from its usefulness in proving the main theorem of this section, is interesting in its own right and depends on the beautiful model theory for a $\Gamma$-contraction developed by Agler and Young in \cite{j.o.t}. They proved, by a Stinespring like method, that if $(S,P)$ is a $\Gamma$-contraction on a Hilbert space $\cl{H}$, then $\cl{H}$ can be isometrically embedded in a Hilbert space $\cl{K}$ (by an isometry $\Pi_{AY}$, say) on which a $\Gamma$-isometry $(\tilde{S},\tilde{P})$ acts such that the isometric image of $\cl{H}$ is a common invariant subspace of $\tilde{S}^*$ and $\tilde{P}^*$ and
$$\Pi_{AY} S^*=\tilde{S}^*|_{\Pi_{AY} \cl{H}},\; \Pi_{AY} P^*=\tilde{P}^*|_{\Pi_{AY} \cl{H}}.$$
Moreover, the $\Gamma$-isometry $(\tilde{S},\tilde{P})$ has a Wold decomposition, viz., $\cl{K}$ has an orthogonal decomposition $\cl{K}_1 \oplus \cl{K}_2$ such that $\cl{K}_1$ and $\cl{K}_2$ reduce both $\tilde{S}$ and $\tilde{P}$, the pair $(\tilde{S}|_{\cl{K}_1}, \tilde{P}|_{\cl{K}_1})$ is a pure $\Gamma$-isometry and
$$ (\tilde{S}_u, \tilde{P}_u) \bydef (\tilde{S}|_{\cl{K}_2}, \tilde{P}|_{\cl{K}_2})$$ is a $\Gamma$-unitary. In addition to this, the structure of a pure $\Gamma$-isometry was completely deciphered by them. It is as follows. There exists a Hilbert space $\cl{E}$ and a bounded operator $Y$ on $\cl{E}$ such that $w(Y) \le 1$ and $(\tilde{S}|_{\cl{K}_1}, \tilde{P}|_{\cl{K}_1})$ is unitarily equivalent to $(T_{\psi} , T_{z})$ acting on $H^2_{\cl{E}}(\mathbb{D})$, where $\psi \in L^{\infty}(\cl{B}(\cl{E}))$ is given by $\psi(z) = Y^* + Yz$ for all $z \in \mathbb{T}$. In short,
\begin{equation}\label{kn}
\Pi_{AY}S^*= \left(
\begin{array}{cc}
M_{Y^*+zY} & 0\\
0 & \tilde{S}_u
\end{array}
\right)^* \Pi_{AY} \text{ and } \Pi_{AY}P^*=\left( \begin{array}{cc}
M_z & 0\\
0 & \tilde{P}_u
\end{array}  \right)^*\Pi_{AY}.
\end{equation}
Let $P$ be a c.n.u. contraction and $\Pi$ be as above. Then in Theorem 4.1 of \cite{JS}, Sarkar showed that there is a unique isometry $\Psi: H^2_{\cl D_{P^*}}(\bb D)\oplus
\overline{\Delta_PL^2_{\cl D_P}(\bb T)} \to \cl{K}_1 \oplus \cl{K}_2$ such that $\Pi_{AY} = \Psi \Pi$. Indeed, $\Psi$ is defined by sending $\Pi h$ to $\Pi_{AY}h$. What Sarkar showed next is significant for our purpose, viz., $\Psi$ is of the form
$(I_{H^2(\mathbb{D})} \otimes \hat V_1) \oplus \hat V_2$, for some isometries $\hat V_1 \in \cl{B}(\cl{D}_{P^*}, \cl{E})$ and $\hat V_2 \in \cl{B}(\overline{\Delta_PL^2_{\cl D_P}(\bb T)},K_2)$.
Taking all this into account, we have from (\ref{kn}),
\begin{eqnarray*}
\Pi S^* & = & \left((I_{H^2(\mathbb{D})} \otimes \hat V^*_1) \oplus \hat V^*_2\right)\left( (I_{H^2(\mathbb{D})} \otimes Y^* + M_z \otimes Y) \oplus \tilde{S}_{u} \right)^*\left((I_{H^2(\mathbb{D})} \otimes \hat V_1) \oplus \hat V_2\right)\Pi
\\
 & = & \left( (I_{H^2(\mathbb{D})} \otimes \hat V^*_1Y^*\hat V_1 + M_z \otimes \hat V^*_1Y\hat V_1) \oplus \hat V^*_2\tilde{S}_{u}\hat V_2 \right)^*\Pi
\end{eqnarray*}
Therefore writing $X=\hat V^*_1Y\hat V_1$ and $R=\hat V^*_2\tilde{S}_{u}\hat V_2$, we get the following neat relation
\begin{eqnarray}\label{gammacont}
\Pi S^*=
\left(
\begin{array}{cc}
M_{X^*+zX} & 0\\
0 & R
\end{array}
\right)^*\Pi
\end{eqnarray}
for some operator $X\in \cl{B}(\cl D_{P^*})$ with $w(X)\le 1$ and $R\in \cl{B}(\overline{\Delta_PL^2_{\cl D_P}(\bb T)})$ such that $(R,M_{e^{it}}|_{\overline{\Delta_PL^2_{\cl D_P}(\bb T)}})$ is a $\Gamma$-unitary on $\overline{\Delta_PL^2_{\cl D_P}(\bb T)}$.
We are going to see that $X$ is unitarily equivalent to the fundamental operator of $(S^*,P^*)$.
 Using (\ref{gammacont}) and (\ref{cont}) we get
\begin{eqnarray*}
S^*-SP^*
&=&
\Pi^*
\left(
\begin{array}{cc}
M_{X^*+z X} & 0\\
0 & R
\end{array}
\right)^*\Pi \\
&-&
\Pi^*
\left(
\begin{array}{cc}
M_{X^*+z X} & 0\\
0 & R
\end{array}
\right)
\Pi \ \Pi^*
\left(
\begin{array}{cc}
M_z & 0\\
0 & M_{e^{it}}
\end{array}
\right)^*\Pi
\\
&=& \Pi^*\
\left(
\begin{array}{cc}
P_{\bb C}\otimes X & 0\\
0 & 0
\end{array}
\right)\Pi  \;\;\text{  [since $(R,M_{e^{it}}|_{\overline{\Delta_PL^2_{\cl D_P}(\bb T)}})$ is a $\Gamma$-unitary.]}
\\
&=& \Pi_0^*\
\left(
\begin{array}{cc}
P_{\bb C}\otimes (V_1XV_1^*) & 0\\
0 & 0
\end{array}
\right)\Pi_0\\
&=& D_{P^*} (V_1XV_1^*) D_{P^*}.
\end{eqnarray*}
Therefore $G=V_1XV_1^*$ is the fundamental operator of
$(S^*,P^*)$. By equation (\ref{gammacont}) we have that $\Pi
\cl{H}=\cl{Q}_P$ is an invariant subspace for $ \left(
\begin{array}{cc}
M_{X^*+zX} & 0\\
0 & R
\end{array}
\right)^*
$.
In other words, $\cl{S}_P = {\cl{Q}_P}^{\perp}$ is invariant under
$\left(
\begin{array}{cc}
M_{X^*+zX} & 0\\
0 & R
\end{array}
\right).$ Hence, using Lemma \ref{Bresult}, we have proved the following.

\begin{lemma} Let $(S,P)$ be a c.n.u. $\Gamma$-contraction. Then there exists $Y\in \cl{B}(\cl D_P)$ with $w(Y)\le 1$ such that
\begin{equation*}
\left(
\begin{array}{cc}
M_{X^*+zX} & 0\\
0 & R
\end{array}
\right)
\left(
\begin{array}{c}
M_{\Theta_P}\\
\Delta_P
\end{array}
\right)
=
\left(
\begin{array}{c}
M_{\Theta_P}\\
\Delta_P
\end{array}
\right)
M_{Y+zY^*},
\end{equation*} where $X$ in the representation of $S$, i.e., Equation (\ref{gammacont}), is unitarily equivalent to the fundamental
operator for $(S^*,P^*).$
\end{lemma}

The following result reveals a beautiful and useful relation
between the operators $S, \ P$ and $P_{\infty},$ when $(S,P)$ is a special $\Gamma$-contraction.

\begin{lemma}
Let $(S,P)$ be a c.n.u. $\Gamma$-contraction such that
$R= M_{e^{it}}+I=M_{e^{it}+I}$ in the representation (\ref{gammacont}) of $S$, then
$$
P_{\infty}^2+PP_{\infty}^2-PP_{\infty}^2S^* = 0.
$$
\end{lemma}
\begin{proof} Let $R=M_{e^{it}+I}.$
Using relations (\ref{cont}), (\ref{*1}), (\ref{gammacont}) and $\Phi\Pi = \Pi_0$
we can write
$$
S=
\Pi_0^*\left(
\begin{array}{cc}
M_{G^*+z G}& 0\\
0 & U^*+I
\end{array}
\right)
\Pi_0
\
\
\text{and}
\
\
P=
\Pi_0^*\left(
\begin{array}{cc}
M_{z}& 0\\
0 & U^*
\end{array}
\right)\Pi_0,
$$
where $G=V_1XV_1^*.$


Consider
\begin{eqnarray*}
P^*+PP^*-PP^*S^*
&=& \Pi_0^*
\left(
\begin{array}{cc}
M_z^* & 0\\
0 & U
\end{array}
\right)\Pi_0
+\Pi_0^*
\left(
\begin{array}{cc}
M_zM_z^* & 0\\
0 & I
\end{array}
\right)\Pi_0\\
&& -
\Pi_0^*
\left(
\begin{array}{cc}
M_zM_z^*M_{G^*+zG}^* & 0\\
0 & U+I
\end{array}
\right)\Pi_0.
\end{eqnarray*}
Applying the definition of $\Pi_0,$ we get
$$
P^*+PP^*-PP^*S^*=P^*+PP^*-PP^*S^*-P_{\infty}^2P^*-P_{\infty}^2
+P_{\infty}^2S^*.$$
Hence,
$P_{\infty}^2P^*+P_{\infty}^2-P_{\infty}^2S^*=0,$ or equivalently,
$P_{\infty}^2+PP_{\infty}^2-PP_{\infty}^2S^*=0$
\end{proof}

We are now in a position to  prove the main result of this section.
\\
\\
\underline{\textit{Proof of Theorem \ref{genthm2}}.} We have seen
that if $(S,P)$ is a c.n.u. $\Gamma$-contraction and $S$ has the
form (\ref{gammacont}), then $S^*-SP^*=D_{P^*}V_1XV_1^*D_{P^*}$
where $X$ is as above. Thus, $V_1XV_1^*$ is the fundamental
operator of $(S^*,P^*).$ Let $G=V_1XV_1^*$ and $F$ denote the
fundamental operator for $(S,P).$ Then by Theorem \ref{genthm}, we
have
\begin{equation}\label{eqn3}
M_{\Theta_P}M_{F+zF^*}=M_{G^*+zG}M_{\Theta_P}.
\end{equation}
We claim that
\begin{equation}\label{eqn4}
M_{e^{it}+I}\Delta_P = \Delta_P M_{F+zF^*}
\end{equation}
As $\Delta_P$ commutes with $M_{e^{it}+I}$ and $\Delta_P$ is non-negative,
therefore Equation (\ref{eqn4}) is equivalent to
\begin{equation}\label{eqn5}
\Delta_P^2M_{e^{it}+I} = \Delta_P^2 M_{F+zF^*}.
\end{equation}

Using the fact that
$$
\Delta_P(t) = [1- \Theta_P(e^{it})^*\Theta_P(e^{it})]^\frac{1}{2}
$$
 and the representation
 $$
 \Theta_P(e^{it}) = [-P+\sum_{n=0}^{\infty}e^{i(n+1)t}D_{P^*}P^{*n}D_P]\big |_{\cl D_P}
$$
we get
\begin{eqnarray}\label{eqn6}
\Delta_P^2M_{e^{it}+I}
&=& D_PPP_{\infty}^2D_P + D_PP_{\infty}^2D_P \nonumber\\
&& + e^{it}[D_PP_{\infty}^2D_P+D_PP_{\infty}^2P^*D_P]\nonumber\\
&& + \sum_{n=2}^{\infty}e^{int}[D_PP_{\infty}^2P^{*(n-1)}D_P+D_PP_{\infty}^2P^{*n}D_P]\nonumber\\
&& + \sum_{n=-\infty}^{-1}e^{int}[D_PP^{1-n}P_{\infty}^2D_P+D_PP^{1-n}P_{\infty}^2P^*D_P]
\end{eqnarray}
and
\begin{eqnarray}\label{eqn7}
\Delta_P^2M_{F+zF^*}
&=& D_P^2F+D_PD_{P^*}GP-D_PSD_P+D_PPP_{\infty}^2S^*D_P \nonumber\\
&& + e^{it}[F^*D_P^2+P^*G^*D_{P^*}D_P-D_PS^*D_P+D_PP_{\infty}^2S^*D_P]\nonumber\\
&& + \sum_{n=2}^{\infty}e^{int}[D_PP_{\infty}^2P^{*(n-1)}S^*D_P\nonumber\\
&& + \sum_{n=-\infty}^{-1}e^{int}[D_PP^{1-n}P_{\infty}^2S^*D_P],
\end{eqnarray}
where to simplify the expressions that appear in the expansion of
$\Delta_P^2M_{F+zF^*}$ we have used that $G$ being the fundamental
operator for $(S^*,P^*)$ satisfies the equations
$D_{P^*}GD_{P^*}=S^*-SP^*$ and
$D_{P^*}S^*=GD_{P^*}+G^*D_{P^*}P^*.$ We defer the proofs of these
two equations till the Appendix. Using these equations, we shall
now show that the coefficients of $e^{int}$ are the same in
Equations (\ref{eqn6}) and \ref{eqn7}). For this, let $L_n$ and
$R_n$ denote the coefficients of $e^{int}$ in the right hand side
of Equations (\ref{eqn6}) and (\ref{eqn7}), respectively.

We first look at
$$ L_0 = D_PPP_{\infty}^2D_P + D_PP_{\infty}^2D_P =
D_PPP_{\infty}^2S^*D_P,$$
since $PP_{\infty}^2+P_{\infty}^2-PP_{\infty}^2S^*=0.$

Now, consider

\begin{eqnarray*}
R_0 &=& D_P^2F+D_PD_{P^*}GP-D_PSD_P+D_PPP_{\infty}^2S^*D_P\\
R_0D_P& = &D_P[D_PFD_P+D_{P^*}GPD_P-SD_P^2+PP_{\infty}^2S^*D_P^2]\\
&=& D_P [S-S^*P+(S^*-SP^*)P-S(1-P^*P)]+ D_PPP_{\infty}^2S^*D_P^2\\
&=& 0+ D_PPP_{\infty}^2S^*D_P^2\\
&=&L_0D_P.
\end{eqnarray*}
Thus $L_0=R_0,$ since $L_0,\ R_0\in \cl{B}(\cl D_P).$

From Equation (\ref{eqn6}),
$$ L_1  =  D_PP_{\infty}^2D_P+D_PP_{\infty}^2P^*D_P  =
D_PP_{\infty}^2S^*D_P,$$
since $P_{\infty}^2+PP_{\infty}^2P^* = P_{\infty}^2S^*.$

Further, from Equation (\ref{eqn7}),
\begin{eqnarray*}
R_1 &=& F^*D_P^2+P^*G^*D_{P^*}D_P-D_PS^*D_P+D_PP_{\infty}^2S^*D_P\\
D_PR_1 &=& D_P[F^*D_P^2+P^*G^*D_{P^*}D_P-D_PS^*D_P+D_PP_{\infty}^2S^*D_P]\\
&=& [D_PF^*D_P+D_PP^*G^*D_{P^*}-D_P^2S^*]D_P+D_P^2P_{\infty}^2S^*D_P\\
&=& [S^*-P^*S+P^*(S^*-SP^*)^*-(1-P^*P)S^*]D_P+ D_P^2P_{\infty}^2S^*D_P\\
&=& D_P^2P_{\infty}^2S^*D_P\\
&=&D_PL_1.
\end{eqnarray*}
Therefore, $D_PR_1=D_PL_1$ which implies that $R_1=L_1,$ as
$R_1, \ L_1\in \cl{B}(\cl D_P).$

We shall now show the equality of $L_n$ and $R_n$ for $n\ge 2.$
\begin{eqnarray*}
L_n &=& D_PP_{\infty}^2P^{*(n-1)}D_P+D_PP_{\infty}^2P^{*n}D_P\\
&=& D_PP_{\infty}^2S^*P^{*(n-1)}D_P = R_n.
\end{eqnarray*}

Lastly, we shall show that $L_n=R_n$ for all $n\le -1.$ For $n\le -1,$
\begin{eqnarray*}
L_n &=& D_PP^{1-n}P_{\infty}^2D_P+D_PP^{1-n}P_{\infty}^2P^*D_P\\
&=&D_PP^{1-n}P_{\infty}^2S^*D_P = R_n.
\end{eqnarray*}

All these above computations show that $L_n = R_n$ for all $n.$
Therefore, $\Delta_P^2M_{e^{it}+I} = \Delta_P^2 M_{F+zF^*}$ which
implies that $M_{e^{it}+I}\Delta_P = \Delta_P M_{F+zF^*}.$ Hence,
Equation (\ref{eqn1}) holds true.

To show the validity of Equation (\ref{eqn2}), note that
$$
\left(
\begin{array}{cc}
M_{X^*+z X}& 0\\
0 & R
\end{array}
\right)^* \Pi(\cl H)\subseteq\Pi(\cl H).
$$
Therefore, by Lemma \ref{Bresult}, we have Equation (\ref{eqn2}).

Conversely, Let $P$ be a c.n.u. contraction on $\cl H,$ and $F, Y\in \cl{B}(\cl D_P)$ with $w(F)\le 1, w(Y)\le 1$
and $G\in G(\cl D_{P^*})$ with $w(G)\le 1,$ satisfy the
Equations (\ref{eqn1}) and (\ref{eqn2}).

Let
$$
S=
\Pi^* \left(
\begin{array}{cc}
M_{X^*+z X}& 0\\
0 & M_{e^{it} + I}
\end{array}
\right)\Pi,
$$
where $X=V_1^*GV_1.$

From Equation (\ref{eqn2}) we can easily deduce that $\Pi(\cl H)$ is invariant
under
$$
\left(
\begin{array}{cc}
M_{X^*+z X}& 0\\
0 & M_{e^{it} + I}
\end{array}
\right)^*,
$$
Also,
$$
P=
\Pi^*\left(
\begin{array}{cc}
M_z& 0\\
0 & M_{e^{it}}
\end{array}
\right)\Pi
\
\
\text{and}
\
\
\left(
\begin{array}{cc}
M_z& 0\\
0 & M_{e^{it}}
\end{array}
\right)^*\Pi(\cl H)\subseteq \Pi(\cl H).
$$
Therefore,
$$
S^*P^* = P^*S^*.
$$
Thus, $(S,P)$ is a commuting pair of bounded operators on $\cl H$ with $\lVert S \rVert \leq 2$.

Now to show that $G$ is the fundamental operator for $(S^*,P^*),$ consider
\begin{eqnarray*}
S^*-SP^* &=& \Pi^*
\left(
\begin{array}{cc}
M_{X^*+z X} & 0\\
0 & M_{e^{it}+I}
\end{array}
\right)^*\Pi \\
&-&
\Pi^*
\left(
\begin{array}{cc}
M_{X^*+z X} & 0\\
0 & M_{e^{it}+I}
\end{array}
\right)
\Pi \ \Pi^*
\left(
\begin{array}{cc}
M_z & 0\\
0 & M_{e^{it}}
\end{array}
\right)^*\Pi
\\
&=& \Pi^*\
\left(
\begin{array}{cc}
P_{\bb C}\otimes X & 0\\
0 & 0
\end{array}
\right)\Pi\\
&=& \Pi_0^*\
\left(
\begin{array}{cc}
P_{\bb C}\otimes G & 0\\
0 & 0
\end{array}
\right)\Pi_0\\
&=& D_{P^*} G D_{P^*}
\end{eqnarray*}
Thus, $S^*-SP^* = D_{P^*} G D_{P^*}.$ Therefore, $G$ is the fundamental operator for $(S^*,P^*).$

Applying the first part of this result to the c.n.u $\Gamma$-contraction
$(S,P),$ we obtain
\begin{eqnarray}\label{eqn8}
\left(
\begin{array}{cc}
M_{G^*+zG} & 0\\
0 & M_{e^{it}+I}
\end{array}
\right)
\left(
\begin{array}{c}
M_{\Theta_P} \\
\Delta_P
\end{array}
\right)
&=&
\left(
\begin{array}{c}
M_{\Theta_P}\\
\Delta_P
\end{array}
\right)
M_{C+zC^*},
\end{eqnarray}
where $C\in \cl{B}(\cl D_P)$ is the fundamental operator for $(S,P).$ Then from the given equation, that is, Equation (2) and Equation (\ref{eqn8}) and the fact that
$$
\left(
\begin{array}{c}
M_{\Theta_P}\\
\Delta_P
\end{array}
\right)
$$
is an isometry we get $M_{F+zF^*}=M_{C+zC^*}.$ Thus $F=C.$ This completes the proof.
\qed

\begin{remark}
Every pure contraction is a c.n.u. contraction. So, for a pure contraction $P\in \cl B(\cl H),$ we have two
results , Theorem
\ref{genthm1} and the converse of Theorem \ref{genthm2}. Theorem \ref{genthm2} demands two conditions, namely Equations
(\ref{eqn1}) and (\ref{eqn2}), for the existence of
$S\in \cl B(\cl H)$ so that the operators $F$ and $G$ are the
fundamental operators for $(S,P)$ and
$(S^*,P^*)$, respectively, whereas in Theorem
\ref{genthm1} the same conclusion holds just by assuming Equation (\ref{eqn1}). Does this make Theorem \ref{genthm2} a weaker result? The answer is no as we shall see from the following discussion that if $P$ is a pure contraction
Equation (\ref{eqn1}) holds if and only if equation (\ref{eqn2}) holds.

Let $P\in\cl B(\cl H)$ be a pure contraction. Then $\P_{\infty}$ and $\Delta_P$ are both zero. Therefore, for the
pure contraction $P,$ Equations (\ref{eqn1}) and (\ref{eqn2}) become
\begin{equation}\label{endremark1}
M_{G^*+zG}M_{\Theta_P} = M_{\Theta_P}M_{F+zF^*}
\end{equation}
and
\begin{equation}\label{endremark2}
M_{G^*+zG}M_{V_1}M_{\Theta_P} = M_{V_1}M_{\Theta_P}M_{Y+zY^*},
\end{equation}
respectively. Further, now since $P$ is pure,
$\Phi = I \otimes V_1, \ \Pi_0 \Pi_0^* + M_{\Theta_P}M_{\Theta_P}^* = I$
and $\Pi_0 = W.$ This implies that $M_{\Theta_P}$ and $(I \otimes V_1)M_{\Theta_P}$  are both isometries in $\cl{B}(H^2_{\cl{D}_P}(\mathbb{D}),H^2_{\cl{D}_{P^*}}(\mathbb{D}))$ and they satisfy the following equation
$$
M_{\Theta_P}M_{\Theta_P}^* = (I \otimes V_1)M_{\Theta_P}M_{\Theta_P}^* (I \otimes V_1^*).
$$
Consequently, $RanM_{\Theta_P} = RanM_{V_1}M_{\Theta_P}.$ Hence, by using Lemma \ref{Bresult}, we can easily conclude that if Equation (\ref{endremark2}) holds, then
Equation (\ref{endremark1}) will also hold. Lastly, if Equation
(\ref{endremark1}) holds, then by using arguments similar to the ones used
in the proof of Lemma \ref{Bresult}, Equation (\ref{endremark2})
will also hold.
\end{remark}

\section{tetrablock contractions}

 In this section, we prove a result for pure tetrablock contractions similar to the result stated in Theorem \ref{genthm} and Theorem \ref{genthm1} for pure $\Gamma$-contractions.

Before we state and prove the main results of this section, we need to recall a result from \cite{sir's tetrablock paper} which will come very handy in proving the main results.
\begin{lemma}\label{tetra}
The fundamental operators $F_1$ and $F_2$ of a tetrablock contraction $(A,B,P)$
are the unique bounded linear operators on $\mathcal{D}_P$ that satisfy the pair
of operator equations
\begin{eqnarray*}
D_PA = X_1D_P + X_2^*D_PP \text{ and } D_P B = X_2D_P + X_1^*D_PP.
\end{eqnarray*}
\end{lemma}
The next two lemmas give analogous results for a tetrablock contraction to the Lemma \ref{adlem2}. These two lemmas can be found in \cite{sau}. We just state the results here without giving the proofs.
\begin{lemma}\label{tetralem4}
Let (A,B,P) be a tetrablock contraction on a Hilbert space $\mathcal{H}$ and $F_1, F_2$ and $G_1,G_2$ be fundamental operators of $(A,B,P)$ and $(A^*,B^*,P^*)$ respectively. Then
$$
PF_i=G_i^*P|_{\mathcal{D}_P}, \text{ for $i$=$1$ and $2$}.
$$
\end{lemma}
\begin{lemma}\label{tetralem3}
Let $(A,B,P)$ be a tetrablock contraction on a Hilbert space $\mathcal{H}$ and $F_1,F_2$ and $G_1,G_2$ be fundamental operators of $(A,B,P)$ and $(A^*,B^*,P^*)$ respectively. Then
\begin{eqnarray*}
&&(F_1^*D_PD_{P^*}-F_2P^*)|_{\mathcal{D}_{P^*}}=D_PD_{P^*}G_1-P^*G_2^* \text{ and }
\\
&&(F_2^*D_PD_{P^*}-F_1P^*)|_{\mathcal{D}_{P^*}}=D_PD_{P^*}G_2-P^*G_1^*.
\end{eqnarray*}
\end{lemma}
The fundamental operators of a tetrablock contraction always abide by two relations (like in the case of $\Gamma$-contractions, Theorem \ref{genthm}). The next theorem, which was proved in \cite{sau}(Corollary 12), gives the relations between them.
\begin{lemma}\label{haha}
Let $F_1$ and $F_2$ be fundamental operators of a tetrablock contraction $(A,B,P)$ and $G_1$ and $G_2$ be fundamental operators of the tetrablock contraction $(A^*,B^*,P^*)$. Then
\begin{eqnarray}\label{adeq2}
&&(F_1^*+F_2z)\Theta_{P^*}(z)=\Theta_{P^*}(z)(G_1+G_2^*z) \text{ and}
\\\label{cccv}
&&(F_2^*+F_1z)\Theta_{P^*}(z)=\Theta_{P^*}(z)(G_2+G_1^*z) \text{ holds} \text{ for all }z \in \mathbb{D}.
\end{eqnarray}
\end{lemma}
\begin{proof}
\begin{eqnarray*}
&&(F_1^*+F_2z)\Theta_{P^*}(z)
\\
&=& (F_1^*+F_2z) (-P^* + \sum_{n=0}^{\infty}z^{n+1}D_PP^nD_{P^*})
\\
&=& (-F_1^*P^*+\sum_{n=1}^{\infty}z^{n}F_1^*D_PP^{n-1}D_{P^*}) + (-zF_2P^* + \sum_{n=2}^{\infty}z^{n}F_2D_PP^{n-2}D_{P^*} )
\\
&=&
-F_1^*P^*+z(-F_2P^*+F_1^*D_PD_{P^*})+ \sum_{n=2}^{\infty} z^n(F_1^*D_PP^{n-1}D_{P^*}+F_2D_PP^{n-2}D_{P^*})
\\
&=& -F_1^*P^*+z(-F_2P^*+F_1^*D_PD_{P^*})+ \sum_{n=2}^{\infty} z^n (F_1^*D_PP+F_2D_P)P^{n-2}D_{P^*}
\\
&=&
-P^*G_1+z(D_PD_{P^*}G_1-P^*G_2^*) + \sum_{n=2}^{\infty} z^nD_PBP^{n-2}D_{P^*} \; [\text{ using Lemma \ref{tetra}, \ref{tetralem4} and \ref{tetralem3}.}]
\end{eqnarray*}
On the other hand
\begin{eqnarray*}
&&\Theta_{P^*}(z)(G_1+G_2^*z)
\\
&=&
(-P^* + \sum_{n=0}^{\infty}z^{n+1}D_PP^nD_{P^*})(G_1+G_2^*z)
\\
&=&
(-P^*G_1 + \sum_{n=1}^{\infty}z^{n}D_PP^{n-1}D_{P^*}G_1) + (-zP^*G_2^*+\sum_{n=2}^{\infty}z^{n}D_PP^{n-2}D_{P^*}G_2^*)
\\
&=&
-P^*G_1 + z(D_PD_{P^*}G_1-P^*G_2^*) + \sum_{n=2}^{\infty}z^{n}(D_PP^{n-1}D_{P^*}G_1+D_PP^{n-2}D_{P^*}G_2^*)
\end{eqnarray*}
\begin{eqnarray*}
&=&
-P^*G_1 + z(D_PD_{P^*}G_1-P^*G_2^*) + \sum_{n=2}^{\infty}z^{n}D_PP^{n-2}(PD_{P^*}G_1+D_PG_2^*)
\\
&=&
-P^*G_1 + z(D_PD_{P^*}G_1-P^*G_2^*) + \sum_{n=2}^{\infty}z^{n}D_PP^{n-2}BD_{P^*}
\\
&=& -P^*G_1+z(D_PD_{P^*}G_1-P^*G_2^*) + \sum_{n=2}^{\infty} z^nD_PBP^{n-2}D_{P^*}.
\end{eqnarray*}
Hence $(F_1^*+F_2z)\Theta_{P^*}(z)=\Theta_{P^*}(z)(G_1+G_2^*z)$ for all $z \in \mathbb{D}$.
Similarly one can prove that $(F_2^*+F_1z)\Theta_{P^*}(z)=\Theta_{P^*}(z)(G_2+G_1^*z) \text{ holds for all }z \in \mathbb{D}$.
\end{proof}

We end with the proof of Theorem \ref{tetrathm}.

{\em{\underline{Proof of Theorem \ref{tetrathm}.}}} The first part is obtained by applying Lemma \ref{haha} to the tetrablock contraction $(A^*,B^*,P^*)$.

For the converse, let $W$ be the isometry defined above. Since $P$ is pure contraction, we have $WP^*=M_z^*W$ as seen in Equation (\ref{contraction}). Equations (\ref{cond}) implies that $(M_{G_1^*+G_2z},M_{G_2^*+G_1z},M_z)$ is a commuting triple of bounded operators on $H^2_{\cl D_{P^*}}(\bb D)$. Using Theorem 5.7 (part (3)) of \cite{sir's tetrablock paper} one can easily check that $(M_{G_1^*+G_2z},M_{G_2^*+G_1z},M_z)$ is actually a tetrablock isometry. Define $A=W^*M_{G_1^*+G_2z}W$ and $B=W^*M_{G_2^*+G_1z}W$. Equations (\ref{adeq3}) and (\ref{help}) tells that $RanM_{\Theta_P}$ is invariant under $M_{G_1^*+G_2z}$ and $M_{G_2^*+G_1z}$. In other words $RanW=(RanM_{\Theta_P})^\perp$ is invariant under $M_{G_1^*+G_2z}^*$ and $M_{G_2^*+G_1z}^*$.
\\
Commutativity of $A$ and $B$ with $P$ can be checked easily. To show that $A$ and $B$ commute, we proceed as follows.
\begin{eqnarray*}
A^*B^*&=&W^*M_{G_1^*+G_2z}^*WW^*M_{G_2^*+G_1z}^*W
\\
&=&
W^*M_{G_1^*+G_2z}^*M_{G_2^*+G_1z}^*W \;\;\;[\text{ since $RanW$ is invariant under $M_{G_2^*+G_1z}^*$.}]
\\
&=&
W^*M_{G_2^*+G_1z}^*M_{G_1^*+G_2z}^*W
\\
&=&
W^*M_{G_2^*+G_1z}^*WW^*M_{G_1^*+G_2z}^*W \;\;\;[\text{ since $RanW$ is invariant under $M_{G_1^*+G_2z}^*$.}]
\\
&=&B^*A^*.
\end{eqnarray*}
Therefore $(A,B,P)$ is a commuting triple of bounded operators. Now we shall show that $(A,B,P)$ is a tetrablock contraction. Note that for every polynomial $f$ in three variables we have $f(A^*,B^*,P^*)=W^*f(T_1^*,T_2^*,T_3^*)W$, where $(T_1,T_2,T_3)=(M_{G_1^*+G_2z},M_{G_2^*+G_1z},M_z)$. Let $f$ be any polynomial in three variables. Then we have
$$
\lVert f(A^*,B^*,P^*) \rVert = \lVert W^*f(T_1^*,T_2^*,T_3^*)W \rVert \leq \lVert f(T_1^*,T_2^*,T_3^*) \rVert \leq \lVert f \rVert_{\bar{E}, \infty}.
$$
Where the last inequality follows from the fact that $(T_1,T_2,T_3)$ is a tetrablock contraction.
\begin{eqnarray*}
A^*-BP^*&=&W^*M_{G_1^*+G_2z}^*W-W^*M_{G_2^*+G_1z}WW^*M_z^*W
\\
&=&
W^*M_{G_1^*+G_2z}^*W - W^*M_{G_2^*+G_1z}M_z^*W    \;\;\;[\text{since $RanW$ is invariant under $M_z^*$}]
\\
&=&
W^*\left( (I \otimes G_1) + (M_z \otimes G_2^*) -(M_z^* \otimes G_2^*) - (M_z M_z^* \otimes G_1) \right)W
\\
&=&
W^*(P_{\mathbb{C} } \otimes G_1) W =D_{P^*}G_1D_{P^*}.
\end{eqnarray*}
Similarly one can show that $B^*-AP^*=D_{P^*}G_2D_{P^*}$.
This shows that $G_1,G_2$ are the fundamental operators of $(A^*,B^*,P^*)$.
Let $X_1,X_2$ be the fundamental operators of $(A,B,P)$. Then we have, by first part of Theorem \ref{tetrathm},
\begin{eqnarray*}
&&(G_1^*+G_2z)\Theta_{P}(z)=\Theta_{P}(z)(X_1+X_2^*z) \text{ and}
\\
&&(G_2^*+G_1z)\Theta_{P}(z)=\Theta_{P}(z)(X_2+X_1^*z) \text{ holds} \text{ for all }z \in \mathbb{D}.
\end{eqnarray*}
By this and the fact that $G_1$ and $G_2$ satisfy Equations (\ref{adeq3}) and (\ref{help}), for some operators $F_1,F_2 \in \mathcal{B}(\mathcal{D}_{P})$ with numerical radii no greater than one, we have $F_1+F_2^*z=X_1+X_2^*z$ and $F_2+F_1^*z=X_2+X_1^*z$, for all $z \in \mathbb{D}$. Which shows that $X_1=F_1$ and $X_2=F_2$. Hence $F_1,F_2$ are the fundamental operators of $(A,B,P)$. This completes the proof of the Theorem.
\qed

\section{appendix}

\subsection{Proof of Equation (\ref{eqn6})}

\begin{eqnarray*}
\Delta_P(t)^2(e^{it}+I) &=& [I-\Theta_P(e^{it})^*\Theta_P(e^{it})][e^{it}+I]\\
&=&[I-(-P^*+\sum_{n=0}^{\infty}e^{-i(n+1)t}D_PP^nD_{P^*})
(-P+\sum_{n=0}^{\infty}e^{i(n+1)t}D_{P^*}P^{*n}D_P)]\\
&&[e^{it}+I]\\
&=& [e^{it}+I]-[P^*+\sum_{n=-\infty}^{-1}e^{int}D_PP^{-n-1}D_{P^*}]\\
&&[-P+e^{it}(D_{P^*}D_P-P)+\sum_{n=2}^{\infty}e^{int}(D_{P^*}P^{*(n-2)}(I+P^*)D_P)]\\
&=& [e^{it}+I]-P^*P-e^{it}(P^*P-P^*D_{P^*}D_P)\\
&+&\sum_{n=2}^{\infty}e^{int}P^*D_{P^*}P^{*(n-2)}(I+P^*)D_P+
\sum_{n=-\infty}^{-1}e^{int}D_PP^{-n-1}D_{P^*}P\\
&-&\sum_{n=-\infty}^{0}e^{int}D_PP^{-n}D_{P^*}(D_{P^*}D_P-P)\\
&-& \sum_{n=-\infty}^{0}e^{int}[\sum_{k=-\infty}^{n-2}D_PP^{-k-1}D_{P^*}^2P^{*(n-k-2)}(I+P^*)D_P]\\
&-&\sum_{n=1}^{\infty}e^{int}[\sum_{k=-\infty}^{-1}D_PP^{-k-1}D_{P^*}^2P^{*(n-k-2)}(I+P^*)D_P]\\
\end{eqnarray*}
We shall now simplify the coefficients of $e^{int}, \ n\in\bb Z.$
Let $C_n$ denote the coefficient of $e^{int}.$ In the following
simplifications we shall be repeatedly using $D_{P^*}^2=I-PP^*, \
D_PP^*=P^*D_{P^*},$ $P_{\infty}^2h = \lim_nP^nP^{*n}h$ for all $h$
and $PP_{\infty}^2P^*=P_{\infty}^2.$
\begin{eqnarray*}
C_0 &=& I-P^*P-D_PD_{P^*}(D_{P^*}D_P-P)-\sum_{k=-\infty}^{-2}D_PP^{-k-1}D_{P^*}^2
P^{*(-k-2)}(I+P^*)D_P\\
&=&D_PPD_P+D_PPP^*D_P-\sum_{k=2}^{\infty}D_PP(P^{k-2}P^{*(k-2)}-P^{k-1}P^{*(k-1)})(I+P^*)D_P\\
&=& D_PPP_{\infty}^2D_P+D_PP_{\infty}^2D_P.
\end{eqnarray*}
\begin{eqnarray*}
C_1 &=& I-P^*P+P^*D_{P^*}D_P-\sum_{k=-\infty}^{-1}D_PP^{-k-1}D_{P^*}^2P^{*(-k-1)}(I+P^*)D_P\\
&=&D_P^2+D_PP^*D_P-\sum_{k=1}^{\infty}D_P(P^{k-1}P^{*(k-1)}-P^{k}P^{*k})(I+P^*)D_P\\
&=&D_PP_{\infty}^2D_P+D_PP_{\infty}^2P^*D_P.
\end{eqnarray*}
Next we look at $C_n, \ n\ge 2.$ For $n\ge 2,$
\begin{eqnarray*}
C_n &=& P^*D_{P^*}P^{*(n-2)}(I+P^*)D_P-\sum_{k=-\infty}^{-1}D_PP^{-k-1}D_{P^*}^2P^{*(n-k-2)}(I+P^*)D_P\\
&=& D_PP^{*(n-1)}D_P+D_PP^{*n}D_P-\sum_{k=1}^{\infty}D_P(P^{k-1}P^{*(k-1)}-P^kP^{*k})P^{*(n-1)}(I+P^*)D_P\\
&=&D_PP_{\infty}^2P^{*(n-1)}D_P+D_PP_{\infty}^2P^{*n}D_P
\end{eqnarray*}
Lastly, we simplify $C_n, \ n\le -1.$ For $n\le -1,$
\begin{eqnarray*}
C_n &=& D_PP^{-n-1}D_{P^*}P-D_PP^{-n}D_{P^*}(D_{P^*}D_P-P)-\sum_{k=-\infty}^{n-2}D_PP^{-k-1}D_{P^*}^2P^{*(n-k-2)}(I+P^*)D_P\\
&=&D_PP^{-n+1}P^*D_P+D_PP^{-n+1}D_P-\sum_{k=0}^{\infty}D_PP^{1-n}(P^kP^{*k}-P^{k+1}P^{*(k+1)})(I+P^*)D_P\\
&=& D_PP^{1-n}P_{\infty}^2D_P+D_PP^{1-n}P_{\infty}^2P^*D_P
\end{eqnarray*}
Thus, Equation (\ref{eqn6}) holds.

\subsection{Proof of Equation (\ref{eqn7})}

\begin{eqnarray*}
\Delta_P(t)^2(F+e^{it}F^*) &=& [I-\Theta_P(e^{it})^*\Theta_P(e^{it})][F+e^{it}F^*]\\
&=& F+e^{it}F^*-\Theta_P(e^{it})^*[G^*+e^{it}G]\Theta_P(e^{it})\\
&&(\text{Since} \  \Theta_P(e^{it})[F+e^{it}F^*]=[G^*+e^{it}G]\Theta_P(e^{it}))\\
&=& F+e^{it}F^*- [-P^*+\sum_{n=0}^{\infty}e^{-i(n+1)t}D_PP^n D_{P^*}][G^*+e^{it}G)]\\
&&[-P+\sum_{n=0}^{\infty}e^{i(n+1)t}D_{P^*}P^{*n}D_P]\\
&=& F+e^{it}F^*- [-P^*+\sum_{n=-\infty}^{-1}e^{int}D_PP^{-n-1}D_{P^*}]\\
&&[-G^*P+e^{it}(G^*D_{P^*}D_P-GP)+\sum_{n=2}^{\infty}e^{int}(G^*D_{P^*}P^*+GD_{P^*})P^{*(n-2)}D_P]\\
&=& F+e^{it}F^*- [-P^*+\sum_{n=-\infty}^{-1}e^{int}D_PP^{-n-1}D_{P^*}]\\
&& [-G^*P+e^{it}(G^*D_{P^*}D_P-GP)+\sum_{n=2}^{\infty}e^{int}D_{P^*}S^*P^{*(n-2)}D_P].
\end{eqnarray*}
To get the last equality we used that $G$ being the fundamental
operator for $(S^*,P^*)$ satisfies
$D_{P^*}S^*=GD_{P^*}+G^*D_{P^*}P^*.$ Next we multiply the last two
terms, as we did to obtain (\ref{eqn6}), and collect coefficients
of $e^{int}.$
\begin{eqnarray*}
\Delta_P(t)^2(F+e^{it}F^*) &=& [F-P^*G^*P-D_PD_P^*(G^*D_{P^*}D_P-GP)\\
&-& \sum_{k=-\infty}^{-2}D_PP^{-k-1}D_{P^*}^2P^{*(-k-2)}S^*D_P]\\
&+&e^{it}[F^*-P^*GP+P^*G^*D_{P^*}D_P-\sum_{k=1}^{\infty}D_PP^{k-1}D_{P^*}^2P^{*(k-1)}S^*D_P]\\
&+&\sum_{n=2}^{\infty}e^{int}[P^*D_{P^*}S^*P^{*(n-2)}D_P-\sum_{k=1}^{\infty}D_PP^{k-1}D_{P^*}^2P^{*(n+k-2)}S^*D_P]\\
&+&\sum_{n=-\infty}^{-1}e^{int}[D_PP^{-n-1}D_{P^*}G^*P-D_PP^{-n}D_{P^*}(G^*D_{P^*}D_P-GP)\\
&-&\sum_{k=2-n}^{\infty}D_PP^{k-1}D_{P^*}^2P^{*(n+k-2)}S^*D_P]
\end{eqnarray*}
Next we simplify the coefficients of $e^{int}, \ n\in\bb Z.$ Let
$D_n$ denote the coefficient of $e^{int}.$ To simplify $D_n's$ we
shall be repeatedly using $D_P^2=I-P^*P, \ D_{P^*}^2=I-PP^*, \
PD_P=D_{P^*}P, \ P^*F= G^*P$ and $D_{P^*}GD_{P^*}=S^*-SP^*.$
\begin{eqnarray*}
D_0 &=& [F-P^*G^*P-D_PD_P^*(G^*D_{P^*}D_P-GP)\\
&-& \sum_{k=-\infty}^{-2}D_PP^{-k-1}D_{P^*}^2P^{*(-k-2)}S^*D_P]\\
&=& F-PP^*F+D_PD_P^*GP-D_PSD_P+D_PPS^*D_P\\
&-& \sum_{k=2}^{\infty}D_PP(P^{k-2}P^{*(k-2)}-P^{k-1}P^{*(k-1)})S^*D_P\\
&=& D_P^2F+D_PD_P^*GP-D_PSD_P+D_PPP_{\infty}^2S^*D_P.
\end{eqnarray*}
\begin{eqnarray*}
D_1&=& F^*-P^*GP+P^*G^*D_{P^*}D_P-\sum_{k=1}^{\infty}D_PP^{k-1}D_{P^*}^2P^{*(k-1)}S^*D_P\\
&=& F^*-F^*P^*P+P^*G^*D_{P^*}D_P-\sum_{k=1}^{\infty}D_P(P^{k-1}P^{*(k-1)}-P^kP^{*k})S^*D_P\\
&=& F^*D_P^2+P^*G^*D_{P^*}D_P-D_PS^*D_P+D_PP_{\infty}^2S^*D_P.
\end{eqnarray*}
For $n\ge 2,$
\begin{eqnarray*}
D_n&=&P^*D_{P^*}S^*P^{*(n-2)}D_P-\sum_{k=1}^{\infty}D_PP^{k-1}D_{P^*}^2P^{*(n+k-2)}S^*D_P\\
&=& P^*D_{P^*}S^*P^{*(n-2)}D_P-\sum_{k=1}^{\infty}D_P(P^{k-1}P^{*(k-1)}-P^kP^{*k})P^{*(n-1)}S^*D_P\\
&=& D_PP_{\infty}^2P^{*(n-1)}S^*D_P.
\end{eqnarray*}
Lastly, for $n\le -1,$
\begin{eqnarray*}
D_n&=& D_PP^{-n-1}D_{P^*}G^*P-D_PP^{-n}D_{P^*}(G^*D_{P^*}D_P-GP)\\
&-&\sum_{k=2-n}^{\infty}D_PP^{k-1}D_{P^*}^2P^{*(n+k-2)}S^*D_P\\
&=& D_PP^{-n-1}D_{P^*}G^*P-D_PP^{-n}(S^*-SP^*)^*D_P+D_PP^{-n}D_{P^*}GP\\
&-&\sum_{k=0}^{\infty}D_PP^{1-n}(P^kP^{*k}-P^{k+1}P^{*(k+1)})S^*D_P\\
&=&D_PP^{1-n}P_{\infty}^2S^*D_P.
\end{eqnarray*}
For each $n\in \bb Z,$ the expression for $D_n$ is same as
required in Equation (\ref{eqn7}). This proves Equation
(\ref{eqn7}).

\end{document}